\documentclass[11pt]{article}

\usepackage{amsfonts}
\usepackage{mathrsfs}
\usepackage{amsmath}
\usepackage{amsthm}
\usepackage{amssymb}
\usepackage{graphicx}
\usepackage{latexsym}
\usepackage{indentfirst}
\usepackage{authblk}
\usepackage[pagewise]{lineno}

\numberwithin{equation}{section}

\newtheorem{thm}{Theorem}[section]
\newtheorem{lema}[thm]{Lemma}

\newtheorem{rmk}[thm]{Remarks}
\newtheorem{cor}[thm]{Corollary}
\newtheorem{definition}[thm]{Definition}
\newtheorem{theorem}[thm]{Theorem}

\title{\Large\bf Global solutions to an initial-boundary value problem 
	for a model of convection driven by surface tension}

\author{\sc  Wending Wu\thanks{Corresponding author. E-mail address: wuwending@mail.bnu.edu.cn} and Xiaojing Xu\\
	{\footnotesize \it Laboratory of Mathematics and Complex Systems (Ministry of Education),}\\
	{\footnotesize \it School of Mathematical Sciences, Beijing Normal University, Beijing 100875, China.}}

\date{}

\begin{document}

\maketitle
\noindent {\bf Abstract.} This paper establishes the global existence of non-negative weak solutions to a two-dimensional, fourth-order nonlinear degenerate parabolic equation modeling surface-tension-driven convection in thin fluid films. First, we construct a regularized approximate problem and prove its solvability via the Galerkin method. Utilizing energy and entropy functionals alongside a singular entropy condition $1/h_0 \in L^1(\Omega)$, we secure uniform a priori bounds for higher-order spatial and time derivatives. These bounds enable the use of the Aubin-Lions lemma and Gagliardo-Nirenberg inequalities to achieve strong compactness and essential $L^6$-integrability. Furthermore, we adopt the Alber-Zhu framework to rigorously define higher-order local weak derivatives and pass to the limit. Finally, we prove the limit function is non-negative, confirming it as a global weak solution to the original problem.
\\[0.2cm]
\\[0.2cm]

\section{Introduction}
Pattern formation and self-organization in non-equilibrium systems are observed in fluid dynamics. Thermal convection, specifically Marangoni-Benard convection driven by surface tension gradients, provides a framework for studying these dynamics. In thin fluid films, microgravity environments, or at free fluid interfaces, flows driven by surface tension dominate bulk buoyancy effects. The driving force acts at the deformable free surface, coupling hydrodynamic stresses with thermal fields. This nonlinear coupling leads to long wave instabilities and surface deformations, which are relevant in applications such as coating processes, microfluidics, and evaporative cooling, where film rupture presents a failure mode.

To investigate these dynamics and rupture phenomena, the three dimensional Navier Stokes equations coupled with heat transfer are reduced via long wave asymptotic methods. The resulting evolution equation is a high order nonlinear degenerate partial differential equation. A model describing the evolution of the film thickness $h$ for a thin liquid layer heated from below was derived by Davis \cite{davis} and discussed in \cite{colinet2001nonlinear}:
\begin{align}\label{laiyuan}
	h_t = -\nabla\cdot\left( \frac{1}{3Ca}h^3\nabla\Delta h - \frac{Ga}{3}h^3\nabla h + \frac{Ma}{2}h^2\nabla h \right),
\end{align}
where $t>0$ denotes time and $x \in \mathbb{R}^2$ represents the spatial coordinates. The positive parameters $Ma$, $Ga$, and $Ca$ correspond to the Marangoni, Galileo, and capillary numbers, respectively. Equation \eqref{laiyuan} captures interfacial dynamics including the formation and coarsening of dry spots \cite{vanhook, vanhook2}.

The equation \eqref{laiyuan} falls within a broader class of fourth-order degenerate parabolic equations of the form:
\begin{align}\label{tuozhan}
	h_t = -\nabla\cdot\left( \theta(h)\nabla\Delta h + \psi(h)\nabla h \right).
\end{align}
Here, the mobility term $\theta(h)$ represents surface tension effects, while $\psi(h)$ accounts for additional forces such as gravity or thermocapillary effects \cite{cankao5,hou12,hou14}.

The analysis of \eqref{tuozhan} presents difficulties due to the lack of a maximum principle and the degeneracy of $\theta$ as $h \to 0$. In one spatial dimension, Bernis and Friedman \cite{cankao4} established the existence of nonnegative weak solutions for the case $\psi \equiv 0$ with $\theta=h^n$. Subsequent studies explored the qualitative behavior of these solutions \cite{cankao3}, self similar source type solutions \cite{cankao5}, and their finite speed of propagation \cite{hou5}. Bertozzi and Pugh \cite{cankao2} incorporated a nonnegative lower order term $\psi \ge 0$, proving the existence of nonnegative weak solutions under assumptions on the ratio $\psi/\theta$.

In multi dimensional settings, the analysis is intricate. For the case $\psi \equiv 0$, Dal Passo et al. \cite{hou7} established the existence of nonnegative solutions using energy and entropy estimates. Gr\"un \cite{cankao6,cankao7} demonstrated existence and finite speed of propagation utilizing generalized Bernis type interpolation inequalities. For models encompassing both the higher order term and the second order nonlinear diffusion term where $\theta=|h|^n$ and $\psi=|h|^m$, existence and propagation properties were investigated by several authors \cite{hou8, cankao1}. These works focus on constructing nonnegative global generalized solutions in unbounded domains or periodic settings.

To fully verify the validity of this model in predicting surface-tension-driven convection and film rupture, rigorous analytical investigations are as essential as numerical simulations. In our previous work \cite{wu}, we established the existence of global weak solutions to the one-dimensional Davis equation under no-flux boundary conditions. However, a rigorous analytical framework for the full two-dimensional case, which is crucial for capturing more realistic physical dynamics, remains incomplete. Motivated by this, the present article aims to extend the existence theory to the two-dimensional setting. Specifically, we investigate the global-in-time existence and non-negativity of weak solutions for the initial-boundary value problem under no-flux conditions.

To formulate this problem mathematically, we simplify the notation by introducing positive constants $\alpha_{1} = \frac{1}{3Ca}$, $\alpha_{2} = \frac{Ga}{3}$, and $\alpha_{3} = \frac{Ma}{2}$. Assume that $\Omega \subset \mathbb{R}^2$ is an open bounded domain with a sufficiently smooth boundary $\partial\Omega$. We denote the space-time cylinder by $Q_{T} := (0,T) \times \Omega$, where $T > 0$ is a given final time, and let $n$ be the unit outward normal vector on $\partial\Omega$. The two-dimensional Davis equation \eqref{laiyuan} can then be rewritten as:
\begin{align}\label{1}
	h_t = -\nabla\cdot\left(\alpha_{1} h^3\nabla\Delta h -\alpha_{2} h^3 \nabla h +\alpha_{3} h^2 \nabla h\right), \qquad \text{in } Q_{T}.
\end{align}
We impose the following initial and boundary conditions
\begin{align} \label{2}
	\frac{\partial h}{\partial n}&=0, \qquad \text{on } (0,T)\times\partial\Omega,
	\\ \label{2.b}
	\alpha_{1} h^3 \frac{\partial \Delta h}{\partial n} -\alpha_{2} h^3 \frac{\partial h}{\partial n} +\alpha_{3} h^2 \frac{\partial h}{\partial n}&=0, \qquad \text{on } (0,T)\times\partial\Omega,
	\\ \label{3}
	h(0,x)&= h_{0}(x), \qquad \text{in } \Omega.
\end{align}
\\
\textbf{Statement of the main result.} Before stating our main results, we introduce essential notations and definitions used throughout this article.
\\
\textbf{Notations.} We denote the $L^2(\Omega)$ norm by $\|\cdot\|$. For a multi index $\alpha$, the weak spatial derivative of order $|\alpha|$ is denoted by $D_x^{\alpha}h$, and $\nabla$ denotes the first order spatial gradient. We employ standard notations for Sobolev and Lebesgue spaces such as $W^{k,p}(\Omega)$ and $L^p(0,T;X)$. The space $W^{-1,\frac{4}{3}}(\Omega)$ denotes the dual space of $W^{1,4}(\Omega)$. For any measurable spatial or space-time domain $Z$ such as $\Omega$, $Q_T$ or their subsets, we denote the $L^2(Z)$ inner product for scalar functions $v_1, v_2$ and vector-valued functions $\mathbf{v}_1, \mathbf{v}_2$ respectively as
$$
(v_1,v_2)_Z = \int_{Z} v_1 v_2, \quad (\mathbf{v}_1,\mathbf{v}_2)_Z = \int_{Z} \mathbf{v}_1 \cdot \mathbf{v}_2,
$$
where all the integrations are performed with respect to the Lebesgue measure. For a space time function $v$ defined on $Q_{T}$, $v(t)$ represents the mapping $x \mapsto v(t,x)$, and we write $v=v(t)$.

For any given $h \in L^2(0,T;H^2(\Omega))$, we define the non-zero sets as:
\begin{align*}
	\mathcal{A}^h &:= \left\{(t,x) \in Q_{T} \mid |h(t,x)|>0\right\},
	\\
	\mathcal{A}^h(t) &:= \{x \in \Omega \mid (t,x) \in \mathcal{A}^h\},\quad t \in (0,T).
\end{align*}
By the standard Sobolev embedding theorem in two dimensions, $h(t) \in H^2(\Omega)$ is spatially continuous for almost every $t$, implying that the spatial slice $\mathcal{A}^h(t)$ is an open set. However, global space-time continuity for $h$ in $Q_T$ is not guaranteed, meaning the space-time set $\mathcal{A}^h$ is not necessarily open. Consequently, standard weak derivatives defined via integration by parts over non-open domains do not mathematically make sense. To overcome this, we utilize a notion of local weak derivatives proposed by Alber and Zhu \cite{alber2008solutions}.

\begin{definition}\label{ruodaoshu}
	Let $\mathcal{A}\subset Q_{T}$ be a domain such that $\mathcal{A}(t)$ is open for almost every $t\in(0,T)$, and let $\alpha$ be a multi-index. A function $g:\mathcal{A}\rightarrow \mathbb{R}$ is called the $\alpha$-th local weak $L^2$-derivative of $h\in L^2(Q_{T})$ with respect to $x$ in $\mathcal{A}$ if:
	\begin{enumerate}
		\item For almost every $t \in (0,T)$, the function $g(t)$ belongs to $L^{2}_{\mathrm{loc}}(\mathcal{A}(t))$ and satisfies
		\begin{align*}
			g(t) = D_x^\alpha h(t)\big|_{\mathcal{A}(t)}
		\end{align*}
		in the usual weak sense;
		\item There exists a sequence $\{\mathcal{A}_n\}_{n=1}^{\infty}$ of measurable subsets $\mathcal{A}_n \subset \mathcal{A}$ with $g\big|_{\mathcal{A}_n}\in L^2(\mathcal{A}_n)$ for all $n\in \mathbb{N}$,
		such that
		\begin{align*}
			\left|\mathcal{A} \setminus \bigcup_{n=1}^\infty \mathcal{A}_n\right| = 0,
		\end{align*}
	    where $|\cdot|$ denotes the Lebesgue measure.
	\end{enumerate}
\end{definition}

Based on this local derivative framework, we are now in a position to rigorously define the global weak solution to the degenerate initial-boundary value problem \eqref{1}-\eqref{3}.

\begin{definition}\label{dingyi1}
	Let $h_0\in L^2(\Omega)$. A function $h=h(t,x)$, which satisfies
	\begin{align}\label{d1}
		h\in L^2(0,T;H^2(\Omega))\cap L^{6}(Q_{T}), \quad \frac{\partial h}{\partial n}=0 \text{ \ a.e. on } \partial \Omega,
	\end{align}
	is a weak solution to the problem \eqref{1}-\eqref{3} if $h$ has the local weak derivative $\nabla \Delta h$ on $\mathcal{A}^h$ in the sense of Definition \ref{ruodaoshu} with $h^3\nabla \Delta h\in L^1(\mathcal{A}^{h})$ and
	\begin{align}\label{d2}
		&(h,\varphi_t)_{Q_{T}}+\alpha_{1} ( h^3 \nabla \Delta h,\nabla \varphi)_{\mathcal{A}^{h}} 
		\nonumber\\
		&=\alpha_{2} (h^3 \nabla h, \nabla \varphi)_{Q_{T}} - \alpha_{3} (h^2 \nabla h, \nabla \varphi)_{Q_{T}}-(h_0,\varphi(0))_{\Omega}
	\end{align}
	holds for all $\varphi\in C_0^\infty((-\infty,T)\times\mathbb{R})$.
\end{definition}

The main result of this article is stated as follows:
\begin{theorem}\label{jieguo}
	Assume that $h_0$ satisfies
	\begin{align}\label{jiashe}
		h_0\in H^1(\Omega), \quad 1/h_0\in L^1(\Omega) \quad \text{and} \quad h_0\geq 0 \ \text{ a.e. in } \Omega.
	\end{align}
	Then, for any given $T>0$, there exists a weak solution $h$ to the problem \eqref{1}-\eqref{3} in the sense of Definition \ref{dingyi1}, such that
	\begin{align}\label{2.8}
		&h\in L^\infty(0,T;H^1(\Omega)) \cap L^2(0,T;H^2(\Omega)),
		\\\label{111}
		&1/{h}\in L^\infty(0,T;L^1(\Omega)),
		\\\label{112}
		&h(t,x) > 0 \quad \text{for a.e. } x \in \Omega, \text{ for all } t \in [0,T].
		\\ \label{2.9}
		&h_t\in L^{\frac{4}{3}}(0,T;W^{-1,\frac{4}{3}}(\Omega)),
		\\ \label{2.10}
		&h^3\nabla \Delta h\in L^{\frac{4}{3}}(Q_{T}),
	\end{align}
    where we set $h^3\nabla \Delta h=0$ on $Q_{T} \setminus \mathcal{A}^h$.
\end{theorem}

\begin{rmk}
	Despite the lack of a maximum principle for fourth-order operators, the non-negativity of $h$ is preserved by the degeneracy of $h^3$ at $h=0$ combined with \eqref{111}.
\end{rmk}

\begin{rmk}
	The condition $1/h_0 \in L^1(\Omega)$ in \eqref{jiashe} does not yield a uniform positive lower bound for $h$. Instead, it provides \eqref{111}, ensuring the strict positivity stated in \eqref{112}.
\end{rmk}

\begin{rmk}
	In Definition \ref{dingyi1}, the integral of $h^3 \nabla \Delta h$ is restricted to $\mathcal{A}^h$. In Theorem \ref{jieguo}, The result \eqref{112} implies the set $Q_{T} \setminus \mathcal{A}^h$ is of Lebesgue measure zero,  which shows that \eqref{2.10} is reasonable.
\end{rmk}

\begin{rmk}
	The uniqueness of the weak solution remains open. The regularity in \eqref{2.8} is insufficient for standard $L^2$-contraction or energy difference estimates on the degenerate term $\nabla \cdot (h^3\nabla\Delta h)$.
\end{rmk}

We now outline the primary mathematical difficulties encountered in proving Theorem \ref{jieguo} and present our analytical strategies:

\textit{First}, the continuous Sobolev embedding $H^1(\Omega) \hookrightarrow L^\infty(\Omega)$ fails in 2D. To pass to the limit in the nonlinear convective terms without imposing restrictive $L^\infty$-assumptions, we leverage the initial condition $1/h_0 \in L^1(\Omega)$ to establish a priori $L^2(Q_T)$-bounds for $D_x^2 h^{\kappa}$. Interpolating this with $h^{\kappa} \in L^\infty(0,T;H^1(\Omega))$ and $h_t^{\kappa} \in L^{\frac{4}{3}}(0,T;W^{-1,\frac{4}{3}}(\Omega))$ via the Aubin-Lions lemma and Gagliardo-Nirenberg inequalities yields a uniform $L^6(Q_T)$-bound. This $L^6$-integrability is precisely sufficient to handle lower-order terms such as $h^3 \nabla h$ and $h^2 \nabla h$.

\textit{Second}, the strong degeneracy of the principal coefficient $h^3$ as $h \to 0$ invalidates global integration by parts, since the positivity set $\mathcal{A}^h$ is not necessarily open in $Q_T$. To rigorously define the highest-order spatial derivatives on non-open domains, we adopt the local weak derivative framework introduced in Definition \ref{ruodaoshu} strictly over $\mathcal{A}^h$.

\textit{Third}, identifying the weak limit of the degenerate term $|h^\kappa|_\kappa^3 \nabla \Delta h^\kappa$ is nontrivial because $\nabla \Delta h$ may blow up as $h \to 0$. We employ a domain-decomposition strategy: on $\mathcal{A}^h$, we apply \textbf{Egorov's theorem} to construct an exhaustion sequence of measurable subsets $\mathcal{A}_n$ with uniform positive lower bounds, securing the local weak limit. On the singular set $\{h=0\}$, we combine H\"older's inequality with the strong $L^6(Q_T)$-convergence of $h^\kappa$ and pass a truncation parameter $\delta \to 0$ to verify that the weak limit vanishes almost everywhere.

\textit{Finally}, fourth-order operators inherently lack a maximum principle. Instead of standard contradiction arguments \cite{cankao4}, we establish strict positivity via a \textbf{double Fatou argument}. By shifting the singular entropy functional to ensure its non-negativity before applying Fatou's lemma spatially, and subsequently extracting a temporal subsequence via the strong continuity $h \in C([0,T]; L^2(\Omega))$, we deduce $h(t,x) > 0$ a.e. for all $t \in [0,T]$. Consequently, the set $\{h=0\}$ has Lebesgue measure zero, which elevates the local weak limit over $\mathcal{A}^h$ to a global one over $Q_T$.

\section{Existence of solutions to the approximate problem}
In this section, we construct an approximate problem of the initial-boundary value problem \eqref{1}-\eqref{3} and establish the existence of its weak solutions via the Galerkin method.

The approximate problem is constructed as follows
\begin{align}\label{a1}
	h_t +\nabla\cdot \left(\alpha_{1} |h|_{\kappa}^3 \nabla \Delta h -\alpha_{2} |h|_{\kappa}^3 \nabla h +\alpha_{3} |h|_{\kappa}^2 \nabla h\right)&=0, \qquad \qquad {\rm in} \,\, Q_{T},
\end{align}
with the initial data and boundary conditions
\begin{align}\label{a2}
	\frac{\partial h}{\partial n}&=0, \qquad \text{on } (0,T)\times\partial\Omega,
	\\ \label{a2b}
	\alpha_{1} |h|_{\kappa}^3 \frac{\partial \Delta h}{\partial n} -\alpha_{2} |h|_{\kappa}^3 \frac{\partial h}{\partial n} +\alpha_{3} |h|_{\kappa}^2 \frac{\partial h}{\partial n}&=0, \qquad \text{on } (0,T)\times\partial\Omega,
	\\ \label{a3}
	h(0,x)&= h_{0}(x), \qquad \text{in } \Omega.
\end{align}
Here 
\begin{align*}
	|h|_{\kappa}:=\sqrt{|h|^2+\kappa^2}
\end{align*}
and $\kappa\in(0,1]$. This definition of $|h|_{\kappa}$ is reasonable because as we will consider the limit as $\kappa\rightarrow 0$, it helps to regularize the original degenerate problem \eqref{1}-\eqref{3}.

We now construct the Galerkin approximate solutions. Let $\{\omega_{i}\}_{i=1}^\infty$ be a complete set of eigenfunctions of the Neumann Laplacian satisfying
\begin{align*}
	-\Delta \omega_{i} &= \lambda_{i} \omega_{i}, \qquad {\rm in} \,\, \Omega,
	\\
	\frac{\partial \omega_{i}}{\partial n} &= 0,  \qquad\quad {\rm on} \,\, \partial\Omega,
\end{align*}
and normalized such that $(\omega_{i},\omega_{j}) = \delta_{ij}$, where $\delta_{ij}$ is the Kronecker delta. For $m \in \mathbb{N_+}$, we define the $m$-th order approximate solution as
\begin{align*}
	h^m(t,x) = \sum_{i=1}^{m}g_{im}(t)\omega_{i}(x),
\end{align*}
where $g_{im}(t) \in \mathbb{R}$ are functions of $t$ to be determined. The initial condition $h_0$ is approximated by the projection
\begin{align*}
	h_0^m(x) = \sum_{i=1}^{m}g_{im}(0)\omega_{i}(x),
\end{align*}
such that $\|h_0^m-h_0\|_{H^1(\Omega)}\rightarrow 0$ as $m \rightarrow \infty$.

By projecting the equation \eqref{a1} onto the finite-dimensional space spanned by $\{\omega_1, \dots, \omega_m\}$, we obtain a system of ordinary differential equations as follows
\begin{align*}
	(h_t^m, \omega_{j}) - \left(\alpha_{1} |h^m|_{\kappa}^3 \nabla \Delta h^m -\alpha_{2} |h^m|_{\kappa}^3 \nabla h^m +\alpha_{3} |h^m|_{\kappa}^2 \nabla h^m, \nabla\omega_{j}\right) = 0, \quad 1 \leq j \leq m.
\end{align*}
Substituting $h^m$ and utilizing $-\Delta \omega_k = \lambda_k \omega_k$, determining $g_{im}(t)$ reduces to solving the following Cauchy problem
\begin{align*}
	\frac{d}{dt}g_{jm}(t) &= F_{j}(g_{1m},\cdots,g_{mm},t),
	\nonumber\\
	g_{jm}(0) &= (h_0, \omega_{j}),
\end{align*}
where $F_j$ is given by
\begin{align*}
	F_{j}(g_{1m},\cdots,g_{mm},t) = &-\alpha_1\sum_{k=1}^{m}\lambda_kg_{km}\int_{\Omega}\left|\sum_{i=1}^{m}g_{im}\omega_i\right|_{\kappa}^3 \nabla\omega_k\cdot\nabla\omega_{j} dx
	\nonumber\\
	&-\alpha_2\sum_{l=1}^{m}g_{lm}\int_{\Omega}\left|\sum_{i=1}^{m}g_{im}\omega_i\right|_{\kappa}^3 \nabla\omega_l\cdot\nabla\omega_{j} dx 
	\nonumber\\
	&+\alpha_3\sum_{l=1}^{m}g_{lm}\int_{\Omega}\left|\sum_{i=1}^{m}g_{im}\omega_i\right|_{\kappa}^2 \nabla\omega_l\cdot\nabla\omega_{j} dx.
\end{align*}
Since $F_j$ is continuous and locally Lipschitz with respect to $g_{im}$, the Picard-Lindelöf theorem for ODEs guarantees the existence of a unique local classical solution $g_{im}(t)$ on an interval $[0, t_m]$.

To prove the global existence of solutions to problem \eqref{a1}-\eqref{a3}, one needs to obtain \textit{a priori} estimates that are uniform with respect to $m$ and valid for all $t \in [0,T]$, as established in the following lemmas. 

\begin{lema}\label{jinsixianyan1}
	There exists a constant $C_\kappa$ independent of $m$, such that the following estimates hold:
	\begin{align}
		\left\| h^m \right\|_{L^\infty (0,T; H^1(\Omega))} &\leq C_\kappa, 
		\\
		\left\| D_x^{2} h^m \right\|_{L^2(Q_T)} &\leq C_\kappa, 
		\\\label{houxuguji}
		\left\| |h^m|_{\kappa}^{\frac{3}{2}} \nabla \Delta h^m \right\|_{L^{2}(Q_T)} &\leq C_\kappa, 
		\\
		\left\| |h^m|_{\kappa}^3 \nabla \Delta h^m \right\|_{L^{\frac{4}{3}}(Q_T)} &\leq C_\kappa.
	\end{align}
\end{lema}

\begin{lema}\label{jinsixianyan2}
	There exists a constant $C_\kappa$ independent of $m$, such that the following estimate holds:
	\begin{align}
		\left\| h_t^m \right\|_{L^{\frac{4}{3}}(0,T; W^{-1,\frac{4}{3}}(\Omega))} \leq C_\kappa.
	\end{align}
\end{lema}

\noindent For the sake of brevity, we omit the proofs of Lemma \ref{jinsixianyan1} and Lemma \ref{jinsixianyan2}, as they are analogous to the corresponding results in Section 3.

\begin{lema}[Aubin-Lions]\label{aubin}
	Let $B_0$, $B$, and $B_1$ be Banach spaces such that $B_0$ and $B_1$ are reflexive, $B_0$ is compactly embedded into $B$, and $B$ is embedded into $B_1$. For $1\leq p_0, p_1 \leq +\infty$, define the space
	$$
	W=\left\{ f \ \middle| \ f\in L^{p_0}(0,T;B_0), \, \frac{\partial f}{\partial t}\in L^{p_1}(0,T;B_1) \right\}.
	$$
	$\mathrm{(1)}$ If $p_0<+\infty$, then the embedding of $W$ into $L^{p_0}(0,T;B)$ is compact. \\
	$\mathrm{(2)}$ If $p_0=+\infty$ and $p_1>1$, then the embedding of $W$ into $C([0,T];B)$ is compact.
\end{lema}
\noindent For a proof of this lemma, we refer to, e.g., \cite{aubin-lions1, aubin-lions2, aubin-lions3}.

By applying Lemmas \ref{jinsixianyan1}, \ref{jinsixianyan2}, and \ref{aubin}, we will establish the following main theorem in this section, which concerns the existence of global solutions to the approximate problem \eqref{a1}-\eqref{a3}.

\begin{theorem}[Global existence]\label{globalexistence}
	Assume that $h_0$ meets assumption \eqref{jiashe}. Then for any given $\kappa>0$, there exists a weak solution $h$ to \eqref{a1}-\eqref{a3}, such that 
	\begin{align}\label{kappajie1}
		\left\| h \right\|_{L^\infty (0,T; H^1(\Omega))} + \left\| h \right\|_{L^2 (0,T; H^3(\Omega))} \leq C_\kappa,
		\\\label{kappajie2}
		\left\| h_t \right\|_{L^{\frac{4}{3}}(0,T; W^{-1,\frac{4}{3}}(\Omega))} \leq C_\kappa.
	\end{align}
\end{theorem}

\begin{proof}
	On the basis of Lemma \ref{jinsixianyan1}, since $|h^m|_\kappa = \sqrt{|h^m|^2+\kappa^2} \geq \kappa > 0$, the estimates \eqref{houxuguji} yields that $\nabla \Delta h^m$ is uniformly bounded in $L^2(Q_T)$ for any fixed $\kappa$. Together with Lemma \ref{jinsixianyan2}, we obtain the following uniform estimates
	\begin{align}\label{yizhixing}
		\left\| h^m \right\|_{L^\infty (0,T; H^1(\Omega))} + \left\| h^m \right\|_{L^2 (0,T; H^3(\Omega))} + \left\| h_t^m \right\|_{L^{\frac{4}{3}}(0,T; W^{-1,\frac{4}{3}}(\Omega))} \leq C_\kappa,
	\end{align}
	where the constant $C_\kappa$ depends on $\kappa$ but is independent of $m$. We now choose a sequence $h_0^m \in C^\infty(\bar{\Omega})$, such that
	\begin{align*}
		\left\| h_0^m-h_0 \right\|_{H^1(\Omega)}\rightarrow 0 \quad \text{as } m\rightarrow \infty.
	\end{align*}
	Applying Lemma \ref{aubin} with $p_0=2$, $p_1=\frac{4}{3}$ and spaces
	\begin{align*}
		B_0=H^3(\Omega), \quad B=H^2(\Omega), \quad B_1=W^{-1,\frac{4}{3}}(\Omega).
	\end{align*}
	Thus \eqref{yizhixing} implies that there exists a subsequence, still denoted by $h^m$, and a limit function $h$ such that
	\begin{align}\label{shoulianH2}
		h^m \rightarrow h \quad \text{strongly in } L^2(0,T; H^2(\Omega)).
	\end{align}
	Furthermore, choosing $B_0=H^1(\Omega)$, $B=L^2(\Omega)$, $p_0=\infty$ and $p_1=\frac{4}{3}$, we conclude that
	\begin{align}\label{shoulianwuqiong}
		h^m \rightarrow h \quad \text{strongly in } C([0,T];L^2(\Omega)).
	\end{align}
	Recall the Gagliardo-Nirenberg inequality in the following form
	\begin{align}\label{nirenberg66}
		\left\| u \right\|_{L^6(\Omega)} \leq C \left\| D_x^2 u \right\|^{\frac{1}{3}} \left\| u \right\|^{\frac{2}{3}} + C \left\| u \right\|.
	\end{align}
	Applying this to $u = h^m - h$ and integrating over time yields
	\begin{align*}
		&\int_{0}^{T} \left\| h^m-h \right\|_{L^6(\Omega)}^6 dt 
		\\
		&\leq C \int_{0}^{T} \left\| D_x^2 h^m - D_x^2 h \right\|^2 \left\| h^m-h \right\|^4 dt + C \int_{0}^{T} \left\| h^m-h \right\|^6 dt 
		\\
		&\leq C \left( \int_{0}^{T}\left\| D_x^2h^m-D_x^2h \right\|^2 dt \right) \left({\rm ess} \sup_{t\in [0,T]}\left\| h^m-h \right\|\right)^4 
		\\
		&\quad + C T \left({\rm ess} \sup_{t\in [0,T]}\left\| h^m-h \right\|\right)^6.
	\end{align*}
	Together with \eqref{shoulianH2} and \eqref{shoulianwuqiong}, one arrives at
	\begin{align}\label{shoulianL6}
		h^m \rightarrow h \quad \text{strongly in } L^6(Q_T).
	\end{align}
	It follows from \eqref{yizhixing} that we can extract a further subsequence such that
	\begin{align*}
		h^m &\overset{*}{\rightharpoonup} h \quad {\rm weakly}^* \quad {\rm in}\quad L^\infty(0,T; H^1(\Omega)), \\
		D_x^3 h^m &\rightharpoonup D_x^3 h \quad {\rm weakly} \quad {\rm in} \quad L^2(Q_T), \\
		h_t^m &\rightharpoonup h_t \quad {\rm weakly} \quad {\rm in} \quad L^{\frac{4}{3}}(0,T; W^{-1,\frac{4}{3}}(\Omega)),
	\end{align*}
	which yield \eqref{kappajie1} and \eqref{kappajie2}. Next, we shall derive convergence estimates to show that $h$ is a solution of \eqref{a1}. We have
	\begin{align*}
		\left| |h^m|_{\kappa}^3 - |h|_{\kappa}^3 \right| &\leq \left| h^m-h \right| \left| |h^m|^2+\kappa^2 +\sqrt{|h^m|^2+\kappa^2}\sqrt{|h|^2+\kappa^2} + |h|^2+\kappa^2 \right|, 
		\\
		\left| |h^m|_{\kappa}^2 - |h|_{\kappa}^2 \right| &\leq \left| h^m-h \right| \left| |h^m|+|h| \right|.
	\end{align*}
	Together with \eqref{shoulianL6} and H\"older's inequality, we obtain
	\begin{align}\label{shoulianL2}
		\left\| |h^m|_{\kappa}^3-|h|_{\kappa}^3 \right\|_{L^2(Q_T)} \rightarrow 0, \quad \text{and} \quad \left\| |h^m|_{\kappa}^2-|h|_{\kappa}^2 \right\|_{L^3(Q_T)} \rightarrow 0.
	\end{align}
	Combining the strong convergence of $|h^m|_{\kappa}^3$ in $L^2(Q_T)$ from \eqref{shoulianL2} with the weak convergence of $\nabla \Delta h^m$ in $L^2(Q_T)$, weak-strong convergence yields
	\begin{align*}
		|h^m|_{\kappa}^3 \nabla \Delta h^m \rightharpoonup |h|_{\kappa}^3 \nabla \Delta h \qquad {\rm weakly} \quad {\rm in} \quad L^1(Q_T).
	\end{align*}
	Lemma \ref{jinsixianyan1} provides a uniform bound for $|h^m|_{\kappa}^3 \nabla \Delta h^m$ in $L^{\frac{4}{3}}(Q_T)$. By the reflexivity of $L^{\frac{4}{3}}(Q_T)$ space, this sequence converge weakly in $L^{\frac{4}{3}}(Q_T)$ up to a subsequence. Uniqueness of weak limits implies
	\begin{align*}
		|h^m|_{\kappa}^3 \nabla \Delta h^m \rightharpoonup |h|_{\kappa}^3 \nabla \Delta h \qquad {\rm weakly} \quad {\rm in} \quad L^{\frac{4}{3}}(Q_T).
	\end{align*}
	Then from the strong convergence of $\nabla h^m$ in $L^2(Q_T)$ implied by \eqref{shoulianH2} and from \eqref{shoulianL2}, we obtain
	\begin{align*}
		| h^m |_{\kappa}^3 \nabla h^m \rightarrow |h|_{\kappa}^3 \nabla h \qquad {\rm strongly} \quad {\rm in} \quad L^1(Q_T).
	\end{align*}
	Finally, \eqref{shoulianH2} and \eqref{shoulianL2} imply
	\begin{align*}
		| h^m |_{\kappa}^2 \nabla h^m \rightarrow |h|_{\kappa}^2 \nabla h \qquad {\rm strongly} \quad {\rm in} \quad L^{\frac{6}{5}}(Q_T).
	\end{align*}
	The above convergence results imply that the global weak solution to \eqref{a1}-\eqref{a3} exists. The proof of Theorem \ref{globalexistence} is complete.
\end{proof}

\section{A priori estimates independent of $\kappa$}
In Section 2, the existence of weak solutions which depend on $\kappa$ was proved. In this section, we will establish the \textit{a priori} estimates for solutions to the problem \eqref{a1}-\eqref{a3}, which are uniform with respect to $\kappa$ for any fixed $T>0$.

\begin{lema}
	There exists a constant $C>0$ independent of $\kappa$, such that for any given $T<\infty$, the following estimate holds:
	\begin{align}\label{4.3}
		\left\| h^\kappa \right\|_{L^\infty (0,T; H^1(\Omega))} \leq C.
	\end{align}
\end{lema}

\begin{proof}
	The derivation of \eqref{4.3} is based on an energy functional $F_{\kappa}[h]$ constructed as follows
	\begin{align}
		\begin{split}\label{ziyounenggai}
			F_{\kappa}[h^{\kappa}]=&\int_{\Omega} \left( \frac{\alpha_1}{2}|\nabla h^{\kappa}|^2+\Psi_\kappa(h^{\kappa}) \right) dx,
			\\
			\Psi_\kappa(h^{\kappa})=&\frac{\alpha_2}{2}(h^{\kappa})^2-\alpha_{3}h^{\kappa}\ln(h^{\kappa}+|h^{\kappa}|_{\kappa})+\alpha_{3}|h^{\kappa}|_{\kappa}+C_1,
		\end{split}
	\end{align}
	where $C_1$ is a sufficiently large positive constant chosen such that
	\begin{align}\label{4.7}
		\Psi_\kappa(y)-\frac{\alpha_2}{4}y^2\geq 0 , \qquad \forall y\in\mathbb{R}, \ \kappa\in(0,1].
	\end{align}
	Then using \eqref{a1}-\eqref{a2b}, \eqref{ziyounenggai} and integrating by parts twice, we obtain
	\begin{align*}
		\frac{dF_{\kappa}}{dt} &= \int_{\Omega} \left( \alpha_1\nabla h^{\kappa} \cdot \nabla h^{\kappa}_t +\Psi_{\kappa}'(h^{\kappa})h^{\kappa}_t \right) dx 
		\nonumber\\
		&= \int_{\Omega} (-\alpha_1\Delta h^{\kappa}+\Psi_{\kappa}'(h^{\kappa}))h^{\kappa}_t dx
		\nonumber\\
		&= \int_{\Omega} (\alpha_1\Delta h^{\kappa}-\Psi_\kappa'(h^{\kappa}))\nabla \cdot \left(\alpha_{1} |h^{\kappa}|_{\kappa}^3 \nabla\Delta h^{\kappa} -\alpha_{2} |h^{\kappa}|_{\kappa}^3 \nabla h^{\kappa} +\alpha_{3} |h^{\kappa}|_{\kappa}^2 \nabla h^{\kappa}\right) dx
		\nonumber\\
		&= -\int_{\Omega} (\alpha_1\nabla\Delta h^{\kappa}-\Psi_\kappa''(h^{\kappa})\nabla h^{\kappa}) \cdot \left(\alpha_{1} |h^{\kappa}|_{\kappa}^3 \nabla\Delta h^{\kappa} -\alpha_{2} |h^{\kappa}|_{\kappa}^3 \nabla h^{\kappa} +\alpha_{3} |h^{\kappa}|_{\kappa}^2 \nabla h^{\kappa}\right) dx
		\nonumber\\
		&= -\int_{\Omega} \frac{1}{|h^{\kappa}|_{\kappa}^3} \left(\alpha_{1} |h^{\kappa}|_{\kappa}^3 \nabla\Delta h^{\kappa} -\alpha_{2} |h^{\kappa}|_{\kappa}^3 \nabla h^{\kappa} +\alpha_{3} |h^{\kappa}|_{\kappa}^2 \nabla h^{\kappa}\right)^2 dx 
		\nonumber\\
		&\leq 0.
	\end{align*}
	Integrating over time from $0$ to $t \in [0,T]$ yields
	\begin{align}\label{4.9}
		F_\kappa[h^{\kappa}(t)] + \int_{0}^{t}\int_{\Omega} \left(\alpha_{1} |h^{\kappa}|_{\kappa}^{\frac{3}{2}} \nabla\Delta h^{\kappa} -\alpha_{2} |h^{\kappa}|_{\kappa}^{\frac{3}{2}} \nabla h^{\kappa} +\alpha_{3} |h^{\kappa}|_{\kappa}^{\frac{1}{2}} \nabla h^{\kappa}\right)^2 dxd\tau = F_\kappa[h_0].
	\end{align}
	By \eqref{jiashe}, \eqref{ziyounenggai} and the Sobolev embedding theorem, one conclude that
	\begin{align}\label{4.10}
		F_\kappa[h^{\kappa}(t)] \leq F_\kappa[h_0] \leq C.
	\end{align}
	Utilizing \eqref{4.7}, the energy functional satisfies the following coercivity property
	\begin{align*}
		F_\kappa[h^{\kappa}(t)] \geq \int_{\Omega} \left( \frac{\alpha_1}{2}|\nabla h^{\kappa}|^2 + \frac{\alpha_2}{4}(h^{\kappa})^2 \right) dx \geq \min\left(\frac{\alpha_1}{2}, \frac{\alpha_2}{4}\right) \left\| h^{\kappa}(t) \right\|_{H^1(\Omega)}^2.
	\end{align*}
	Combining this lower bound with \eqref{4.10}, one arrives at \eqref{4.3}.
\end{proof}

We next turn to higher-order \textit{a priori} estimates. The following lemma bounds the third-order nonlinear terms by the $L^2$-norm of $D_x^2 h^{\kappa}$.

\begin{lema}
	There exists a constant $C>0$ independent of $\kappa$, such that for any given $T<\infty$, the following estimates hold:
	\begin{align}\label{guji2}
		 \left\| | h^{\kappa} |_{\kappa}^{\frac{3}{2}} \nabla\Delta h^{\kappa} \right\|_{L^2(Q_T)} &\leq C\left( 1 + \left\| D_x^2 h^{\kappa} \right\|_{L^2(Q_T)}^2 \right)^{\frac{1}{2}},
		\\\label{guji3}
		\left\| |h^{\kappa}|_{\kappa}^3\nabla\Delta h^{\kappa} \right\|_{L^{\frac{4}{3}}(Q_T)} &\leq C \left(1 + \left\| D_x^2 h^{\kappa} \right\|_{L^2(Q_T)}^2 \right)^{\frac{3}{4}}.
	\end{align}
\end{lema}

\begin{proof}
	To estimate \eqref{guji2}, we recall the Gagliardo-Nirenberg inequality in the following form
	\begin{align}\label{Nirenbergwuqiong}
		\left\| h^{\kappa} \right\|_{L^\infty(\Omega)} \leq C \left\| D_x^2 h^{\kappa} \right\|^{\frac{1}{2}} \left\| h^{\kappa} \right\|^{\frac{1}{2}} + C\left\| h^{\kappa} \right\|.
	\end{align}
	Utilizing \eqref{4.9}, \eqref{4.10} and applying the inequality $\frac{1}{2}A^2 \leq (A-B)^2 + B^2$, we obtain
	\begin{align*}
		&\frac{\alpha_{1}}{2} \int_{0}^{t}\int_{\Omega} |h^{\kappa}|_{\kappa}^3|\nabla\Delta h^{\kappa}|^2 dxd\tau
		\nonumber\\
		&\leq \int_{0}^{t}\int_{\Omega}\left(\alpha_{1} |h^{\kappa}|_{\kappa}^{\frac{3}{2}} \nabla\Delta h^{\kappa} -\alpha_{2} |h^{\kappa}|_{\kappa}^{\frac{3}{2}} \nabla h^{\kappa} +\alpha_{3} |h^{\kappa}|_{\kappa}^{\frac{1}{2}} \nabla h^{\kappa}\right)^2 dxd\tau
		\nonumber\\
		&\quad + \int_{0}^{t}\int_{\Omega}\left( -\alpha_{2} |h^{\kappa}|_{\kappa}^{\frac{3}{2}} \nabla h^{\kappa} +\alpha_{3} |h^{\kappa}|_{\kappa}^{\frac{1}{2}} \nabla h^{\kappa}\right)^2 dxd\tau
		\nonumber\\
		&\leq C + C \int_{0}^{t}\int_{\Omega} \left( |h^{\kappa}|_{\kappa}^{3} |\nabla h^{\kappa}|^2 + |h^{\kappa}|_{\kappa} |\nabla h^{\kappa}|^2 \right) dxd\tau
		\nonumber\\
		&\leq C + C \int_{0}^{t} \left\| \nabla h^{\kappa} \right\|^2 \left( \left\| h^{\kappa} \right\|_{L^\infty(\Omega)}^3 + \left\| h^{\kappa} \right\|_{L^\infty(\Omega)} + 1 \right) d\tau.
	\end{align*}
	Substituting \eqref{Nirenbergwuqiong} into the above inequality, using \eqref{4.3} and Young's inequality yields
	\begin{align*}
		&\frac{\alpha_{1}}{2} \int_{0}^{t}\int_{\Omega} |h^{\kappa}|_{\kappa}^3|\nabla\Delta h^{\kappa}|^2 dxd\tau 
		\nonumber\\
		&\leq C + C \int_{0}^{t} \left\| \nabla h^{\kappa} \right\|^2 \left( \left\| D_x^2 h^{\kappa} \right\|^{\frac{3}{2}} \left\| h^{\kappa} \right\|^{\frac{3}{2}} + \left\| h^{\kappa} \right\|^3 + \left\| D_x^2 h^{\kappa} \right\|^{\frac{1}{2}} \left\| h^{\kappa} \right\|^{\frac{1}{2}} + \left\| h^{\kappa} \right\| + 1 \right) d\tau
		\nonumber\\
		&\leq C + C \int_{0}^{t} \left( \left\| D_x^2 h^{\kappa} \right\|^{\frac{3}{2}} + \left\| D_x^2 h^{\kappa} \right\|^{\frac{1}{2}} + 1 \right) d\tau
		\nonumber\\
		&\leq C + C \int_{0}^{t} \left\| D_x^2 h^{\kappa} \right\|^2 d\tau,
	\end{align*}
	which yields \eqref{guji2}. Now we are going to estimate \eqref{guji3}, which follows from \eqref{guji2} and the Gagliardo-Nirenberg inequality for $L^r$ spaces for $r\geq 2$:
	\begin{align}\label{NirenbergR}
		\left\| h^{\kappa} \right\|_{L^r(\Omega)} \leq C \left\| D_x^2 h^{\kappa} \right\|^{\frac{1}{2}-\frac{1}{r}} \left\| h^{\kappa} \right\|^{\frac{1}{2}+\frac{1}{r}} + C \left\| h^{\kappa} \right\|.
	\end{align}
	Choosing $1<p<2$ and $pq=2$ with $\frac{1}{q}+\frac{1}{q'}=1$, H\"older's inequality yields
	\begin{align}\label{2.72}
		&\int_{0}^{t}\int_{\Omega}\left(|h^{\kappa}|_{\kappa}^3|\nabla\Delta h^{\kappa}|\right)^p dxd\tau
		\nonumber\\
		&= \int_{0}^{t}\int_{\Omega} |h^{\kappa}|_{\kappa}^{\frac{3p}{2}} \left(|h^{\kappa}|_{\kappa}^{\frac{3}{2}}|\nabla\Delta h^{\kappa}|\right)^p dxd\tau
		\nonumber\\
		&\leq \left(\int_{0}^{t}\int_{\Omega}|h^{\kappa}|_{\kappa}^{\frac{3pq'}{2}}dxd\tau\right)^{\frac{1}{q'}} \left(\int_{0}^{t}\int_{\Omega}|h^{\kappa}|_{\kappa}^{\frac{3pq}{2}}|\nabla\Delta h^{\kappa}|^{pq}dxd\tau\right)^{\frac{1}{q}}
		\nonumber\\
		&\leq C \left(\int_{0}^{t}\int_{\Omega}|h^{\kappa}|_{\kappa}^{\frac{3p}{2-p}}dxd\tau\right)^{\frac{2-p}{2}} \left(\int_{0}^{t}\int_{\Omega}|h^{\kappa}|_{\kappa}^3|\nabla\Delta h^{\kappa}|^2dxd\tau\right)^{\frac{p}{2}}.
	\end{align}
	Using \eqref{4.3} and \eqref{NirenbergR} with $r = \frac{3p}{2-p}$, we estimate the first integral as follows
	\begin{align}\label{2.73}
		\int_{0}^{t}\int_{\Omega}|h^{\kappa}|_{\kappa}^{\frac{3p}{2-p}}dxd\tau
		&\leq \int_{0}^{t} \left\| |h^{\kappa}|+\kappa \right\|_{L^{\frac{3p}{2-p}}(\Omega)}^{\frac{3p}{2-p}} d\tau
		\nonumber\\
		&\leq C \int_{0}^{t} \left\| h^{\kappa} \right\|_{L^{\frac{3p}{2-p}}(\Omega)}^{\frac{3p}{2-p}} d\tau + C
		\nonumber\\
		&\leq C \int_{0}^{t} \left( \left\| D_x^2 h^{\kappa} \right\|^{\frac{5p-4}{4-2p}} \left\| h^{\kappa} \right\|^{\frac{4+p}{4-2p}} + \left\| h^{\kappa} \right\|^{\frac{3p}{2-p}} \right) d\tau + C
		\nonumber\\
		&\leq C \int_{0}^{t} \left\| D_x^2 h^{\kappa} \right\|^2 d\tau + C,
	\end{align}
	for $\frac{5p-4}{4-2p} \leq 2$, equivalent to $p \leq \frac{4}{3}$. Taking $p = \frac{4}{3}$, estimates \eqref{guji2}, \eqref{2.72} and \eqref{2.73} yield \eqref{guji3}. The proof of this lemma is thus complete.
\end{proof}

By introducing an entropy functional, the following lemma yields a uniform $L^2(Q_T)$ bound for $D_x^2h^{\kappa}$.

\begin{lema}\label{entropy}
	There exists a constant $C>0$ independent of $\kappa$, such that for any given $T<\infty$, the following estimate holds:
	\begin{align}\label{D2}
		\left\| D_x^2h^{\kappa} \right\|_{L^2(Q_T)} \leq C.
	\end{align}
\end{lema}

\begin{proof}
We introduce the auxiliary functions $g_\kappa(s)$ and $G_\kappa(s)$ defined by
\begin{align*}
	g_\kappa(s) = \int_{A}^{s} \frac{1}{|r|_{\kappa}^3} dr, \qquad G_\kappa(s) = \int_{A}^{s} g_\kappa(t) dt = \int_{A}^{s}\int_{A}^{t} \frac{1}{|r|_{\kappa}^3} dr dt,
\end{align*}
where $A>0$ is a fixed constant. Differentiation yields
\begin{align*}
	G_\kappa'(s) = g_\kappa(s), \quad G_\kappa''(s) = g_\kappa'(s) = \frac{1}{|s|_{\kappa}^3}.
\end{align*}
Since $G_\kappa(A) = G_\kappa'(A) = 0$ and $G_\kappa''(s) > 0$, it is strictly convex and satisfies $G_\kappa(s) \geq 0$ for $s \in \mathbb{R}$. We compute $G_\kappa(s)$ as
\begin{align*}
	G_\kappa(s)&= \int_{A}^{s} \left( \frac{t}{\kappa^2\sqrt{t^2+\kappa^2}} - \frac{A}{\kappa^2\sqrt{A^2+\kappa^2}} \right) dt
	\nonumber\\
	&= \frac{\sqrt{s^2+\kappa^2}}{\kappa^2} - \frac{As}{\kappa^2\sqrt{A^2+\kappa^2}} - \frac{1}{\sqrt{A^2+\kappa^2}}.
\end{align*}
Taking the limit as $\kappa \rightarrow 0$, we obtain the classical entropy function $G_0(s)$:
\begin{align}\label{g}
	G_0(s)=\left\{
	\begin{aligned}
		&\frac{s}{2A^2}+\frac{1}{2s}-\frac{1}{A}, \quad &s > 0,\\
		&\infty, \quad &s\leq 0.
	\end{aligned}\right.
\end{align}
The inequality $0 \leq G_\kappa(s) \leq G_0(s)$ holds for $s \in \mathbb{R}$. Testing \eqref{a1} by $g_\kappa(h^{\kappa})$, integrating over $Q_t$ and applying integration by parts with \eqref{a2b} yields
\begin{align}\label{6.7}
	&\int_{\Omega}G_\kappa(h^{\kappa}(t,x))dx + \int_{0}^{t}\int_{\Omega} \left( \alpha_{1}|\Delta h^{\kappa}|^2 + \alpha_{2}|\nabla h^{\kappa}|^2 \right) dxd\tau
	\nonumber\\
	&= \int_{\Omega}G_\kappa(h_0(x))dx + \int_{0}^{t}\int_{\Omega} \alpha_{3}\frac{|\nabla h^{\kappa}|^2}{|h^{\kappa}|_\kappa} dxd\tau.
\end{align}
By \eqref{jiashe} and \eqref{g}, we obtain
\begin{align}\label{G0}
	\int_{\Omega} G_\kappa(h_0(x)) dx \leq \int_{\Omega} G_{0}(h_0(x)) dx \leq C.
\end{align}
To estimate the last term in \eqref{6.7}, we introduce a smooth function $L_\kappa(s) = \int_A^s \frac{1}{|r|_\kappa} dr$ satisfying $\nabla L_\kappa(h^{\kappa}) = \frac{\nabla h^{\kappa}}{|h^{\kappa}|_\kappa}$. Applying integration by parts and Young's inequality, one conclude that
\begin{align}\label{2.6.11}
	\int_{0}^{t}\int_{\Omega} \alpha_{3}\frac{|\nabla h^{\kappa}|^2}{|h^{\kappa}|{\kappa}} dxd\tau
	&= \int_{0}^{t}\int_{\Omega} \alpha_{3} \nabla h^{\kappa} \cdot \nabla L_\kappa(h^{\kappa}) dxd\tau
	\nonumber\\
	&= - \int_{0}^{t}\int_{\Omega} \alpha_{3} \Delta h^{\kappa} L_\kappa(h^{\kappa}) dxd\tau
	\nonumber\\
	&\leq \int_{0}^{t}\int_{\Omega} \frac{\alpha_1}{2} |\Delta h^{\kappa}|^2 dxd\tau + C_{\alpha_1} \int_{0}^{t}\int_{\Omega} |L_\kappa(h^{\kappa})|^2 dxd\tau.
\end{align}
Notice that $L_\kappa(s)$ exhibits a logarithmic growth as $|s| \rightarrow 0$, whereas $G_\kappa(s)$ grows as $\frac{1}{|s|_\kappa}$. The inverse growth dominates the logarithmic square growth, implying there exists a constant $C$, independent of $\kappa$, such that $C_{\alpha_1} |L_\kappa(s)|^2 \leq  G_\kappa(s) + C$ for $s \in \mathbb{R}$. Thus we have
\begin{align}\label{Lk}
	C_{\alpha_1} \int_{0}^{t}\int_{\Omega} |L_\kappa(h^{\kappa})|^2 dxd\tau \leq  \int_{0}^{t}\int_{\Omega} G_\kappa(h^{\kappa}) dxd\tau + C.
\end{align}
Substituting \eqref{G0}, \eqref{2.6.11} and \eqref{Lk} into \eqref{6.7}, one arrive at
\begin{align*}
	\int_{\Omega}G_\kappa(h^{\kappa}(t,x))dx + \frac{\alpha_1}{2} \int_{0}^{t}\int_{\Omega} |\Delta h^{\kappa}|^2 dxd\tau \leq C + \int_{0}^{t}\int_{\Omega} G_\kappa(h^{\kappa}(\tau,x)) dxd\tau.
\end{align*}
Applying Gronwall's lemma to $\int_{\Omega}G_\kappa(h^{\kappa})dx$ yields
\begin{align}\label{6.15}
	\int_{\Omega}G_\kappa(h^{\kappa}(t,x))dx \leq C, \quad \text{and} \quad \int_{0}^{t}\int_{\Omega}|\Delta h^{\kappa}|^2dxd\tau \leq C.
\end{align}
The combination of this uniform $L^2(Q_T)$ bound for $\Delta h^{\kappa}$ and standard elliptic estimates leads to \eqref{D2}, which completes the proof of this lemma.
\end{proof}


\begin{cor}\label{cor}
		There exists a constant $C$ independent of $\kappa$, such that for any given $T<\infty$, the following estimates hold:
	\begin{align}\label{guji32}
		\left\| | h^{\kappa} |_{\kappa}^{\frac{3}{2}} \nabla\Delta h^{\kappa} \right\|_{L^2(Q_T)} &\leq C,
		\\\label{guji43}
		\left\| |h^{\kappa}|_{\kappa}^3 \nabla\Delta h^{\kappa} \right\|_{L^{\frac{4}{3}}(Q_T)} &\leq C.
	\end{align}
\end{cor}

\begin{proof}
	Estimates \eqref{guji32} and \eqref{guji43} follow from substituting \eqref{D2} into \eqref{guji2} and \eqref{guji3}.
\end{proof}


\begin{lema}
	There exists a constant $C>0$ independent of $\kappa$, such that for any given $T<\infty$, the following estimate holds:
	\begin{align}\label{lemma4.2}
		\left\| h^{\kappa}_t \right\|_{L^{\frac{4}{3}}(0,T; W^{-1,\frac{4}{3}}(\Omega))} \leq C.
	\end{align}
\end{lema}

\begin{proof}
	Since $\Omega$ is a bounded domain, H\"older's inequality implies $\left\| \nabla h^{\kappa} \right\|_{L^{\frac{4}{3}}(\Omega)} \leq C \left\| \nabla h^{\kappa} \right\|$. Applying this with the uniform bounds \eqref{4.3}, \eqref{D2} and the Gagliardo-Nirenberg inequality \eqref{Nirenbergwuqiong} yields
	\begin{align}\label{ht1}
		\int_{0}^{T}\int_{\Omega} |h^{\kappa}|_{\kappa}^4 |\nabla h^{\kappa}|^{\frac{4}{3}} dxdt
		&\leq C \int_{0}^{T} \left( \left\| h^{\kappa} \right\|_{L^\infty(\Omega)}^4 + 1 \right) \left\| \nabla h^{\kappa} \right\|_{L^{\frac{4}{3}}(\Omega)}^{\frac{4}{3}} dt \nonumber\\
		&\leq C \int_{0}^{T} \left( \left\| h^{\kappa} \right\|_{L^\infty(\Omega)}^4 + 1 \right) \left\| \nabla h^{\kappa} \right\|^{\frac{4}{3}} dt \nonumber\\
		&\leq C \int_{0}^{T} \left( \left\| D_x^2 h^{\kappa} \right\|^2 \left\| h^{\kappa} \right\|^2 \left\| \nabla h^{\kappa} \right\|^{\frac{4}{3}} + \left\| h^{\kappa} \right\|^4 \left\| \nabla h^{\kappa} \right\|^{\frac{4}{3}} + \left\| \nabla h^{\kappa} \right\|^{\frac{4}{3}} \right) dt \nonumber\\
		&\leq C \int_{0}^{T} \left( \left\| D_x^2 h^{\kappa} \right\|^2 + 1 \right) dt \nonumber\\
		&\leq C.
	\end{align}
	For the lower-order term, we obtain
	\begin{align}\label{ht2}
		\int_{0}^{T}\int_{\Omega} |h^{\kappa}|_{\kappa}^{\frac{8}{3}} |\nabla h^{\kappa}|^{\frac{4}{3}} dxdt
		&\leq C \int_{0}^{T} \left( \left\| h^{\kappa} \right\|_{L^\infty(\Omega)}^{\frac{8}{3}} + 1 \right) \left\| \nabla h^{\kappa} \right\|_{L^{\frac{4}{3}}(\Omega)}^{\frac{4}{3}} dt \nonumber\\
		&\leq C \int_{0}^{T} \left( \left\| h^{\kappa} \right\|_{L^\infty(\Omega)}^{\frac{8}{3}} + 1 \right) \left\| \nabla h^{\kappa} \right\|^{\frac{4}{3}} dt \nonumber\\
		&\leq C \int_{0}^{T} \left( \left\| D_x^2 h^{\kappa} \right\|^{\frac{4}{3}} \left\| h^{\kappa} \right\|^{\frac{4}{3}} \left\| \nabla h^{\kappa} \right\|^{\frac{4}{3}} + \left\| h^{\kappa} \right\|^{\frac{8}{3}} \left\| \nabla h^{\kappa} \right\|^{\frac{4}{3}} + \left\| \nabla h^{\kappa} \right\|^{\frac{4}{3}} \right) dt \nonumber\\
		&\leq C \int_{0}^{T} \left( \left\| D_x^2 h^{\kappa} \right\|^{\frac{4}{3}} + 1 \right) dt \nonumber\\
		&\leq C.
	\end{align}
	For any test function $\varphi \in C_{0}^{\infty}(0,T; C^\infty(\bar{\Omega}))$, the weak formulation of \eqref{a1} and estimates \eqref{guji43}, \eqref{ht1} and \eqref{ht2} yield
	\begin{align*}
		&\left| \int_{0}^{T} \langle h^{\kappa}_t, \varphi \rangle dt \right| 
		\nonumber\\
		&= \left| \iint_{Q_T} \left( \alpha_{1}|h^{\kappa}|_{\kappa}^3 \nabla\Delta h^{\kappa} - \alpha_{2} |h^{\kappa}|_{\kappa}^3 \nabla h^{\kappa} + \alpha_{3} |h^{\kappa}|_{\kappa}^2 \nabla h^{\kappa} \right) \cdot \nabla \varphi \,dxdt \right| \nonumber\\
		&\leq \left\| \alpha_{1}|h^{\kappa}|_{\kappa}^3 \nabla\Delta h^{\kappa} - \alpha_{2} |h^{\kappa}|_{\kappa}^3 \nabla h^{\kappa} + \alpha_{3} |h^{\kappa}|_{\kappa}^2 \nabla h^{\kappa} \right\|_{L^{\frac{4}{3}}(Q_T)} \left\| \nabla \varphi \right\|_{L^4(Q_T)} \nonumber\\
		&\leq C \left( \left\| |h^{\kappa}|_{\kappa}^3 \nabla\Delta h^{\kappa} \right\|_{L^{\frac{4}{3}}(Q_T)} + \left\| |h^{\kappa}|_{\kappa}^3 \nabla h^{\kappa} \right\|_{L^{\frac{4}{3}}(Q_T)} + \left\| |h^{\kappa}|_{\kappa}^2 \nabla h^{\kappa} \right\|_{L^{\frac{4}{3}}(Q_T)} \right) \left\| \nabla \varphi \right\|_{L^4(Q_T)} \nonumber\\
		&\leq C \left\| \varphi \right\|_{L^4(0,T; W^{1,4}(\Omega))}.
	\end{align*}
	By a density argument, the above inequality implies \eqref{lemma4.2} and completes the proof of the lemma.
\end{proof}

\section{Non-negativity and Existence of solutions to the original problem}

In this section, we utilize the \textit{a priori} estimates independent of $\kappa$ established in Section 3 to study the convergence of $h^\kappa$ as $\kappa\rightarrow 0$. Subsequently, we will prove the non-negativity of limit function and complete the proof of Theorem \ref{jieguo}. 
Before proceeding to the main results of this section, we recall Egorov's theorem, which bridges almost everywhere convergence and uniform convergence.

\begin{theorem}[Egorov]\label{shibian}
	Let $(\Gamma,\Sigma,\mu)$ be a measure space with $\mu(\Gamma)<\infty$, and let $\{f_j\}_{j\in \mathbb{N}}$ and $f$ be real-valued, measurable functions on $\Gamma$. Suppose $f_j(x)\rightarrow f(x)$ as $j\rightarrow\infty$ for almost every $x\in\Gamma$.
	
	Then, for every $\varepsilon>0$, there exists a measurable subset $M_\varepsilon\subset\Gamma$ with $\mu(\Gamma \setminus M_\varepsilon)<\varepsilon$, such that $f_j(x)$ converges to $f(x)$ uniformly on $M_\varepsilon$. That is, for every $\delta>0$, there exists an integer $N_\delta > 0$ such that for all $j>N_\delta$ and all $x\in M_\varepsilon$, there holds
	\begin{align*}
		|f_j(x)-f(x)|<\delta.
	\end{align*}
\end{theorem}
\noindent A proof of Theorem \ref{shibian} can be found, for example, in \cite[p.16]{shi}.

With the uniform \textit{a priori} estimates established in the previous section, we pass to the limit as $\kappa \to 0$. The following lemma extracts convergent subsequences of $h^\kappa$ and establishes their convergence properties.
\begin{lema}\label{shoulian1}
	Let $0<\alpha<1$. There exists a subsequence $\kappa_n\rightarrow 0$ and a limit function $h \in L^2(0,T;C^\alpha(\bar{\Omega})) \cap C([0,T];L^2(\Omega))$ satisfying
	\begin{align}\label{5.1}
		h \in L^\infty(0,T;H^1(\Omega)) \cap L^2(0,T;H^2(\Omega)), \qquad h_t \in L^{\frac{4}{3}}(0,T;W^{-1,\frac{4}{3}}(\Omega)),
	\end{align}
	such that the sequence $h^{\kappa_n}$, still denoted by $h^{\kappa}$ for simplicity, satisfies the following convergence properties as $\kappa \rightarrow 0$:
	\begin{alignat}{2}
		\label{5.2}
		h^\kappa &\rightarrow h \quad && \text{strongly in } L^2(0,T;C^\alpha(\bar{\Omega})), \\
		\label{5.2b}
		h^\kappa(t) &\rightarrow h(t) \quad && \text{strongly in } C^\alpha(\bar{\Omega}) \text{ for a.e. } t \in (0,T), \\
		\label{5.3}
		h^\kappa &\rightarrow h \quad && \text{strongly in } C([0,T];L^2(\Omega)), \\
		\label{kappa6}
		|h^\kappa|_{\kappa} &\rightarrow |h| \quad && \text{strongly in } L^6(Q_{T}), \\
		\label{5.4}
		|h^\kappa|_{\kappa}^3 &\rightarrow |h|^3 \quad && \text{strongly in } L^2(Q_{T}), \\
		\label{5.5}
		|h^\kappa|_{\kappa}^2 &\rightarrow |h|^2 \quad && \text{strongly in } L^3(Q_{T}), \\
		\label{5.55}
		\nabla h^\kappa &\rightharpoonup \nabla h \quad && \text{weakly \, in } L^2(Q_{T}).
	\end{alignat}
\end{lema}

\begin{proof}
	Applying Lemma \ref{aubin} with $p_0=2$, $p_1=\frac{4}{3}$ and spaces
	\begin{align*}
		B_0=H^2(\Omega),\quad B=C^{\alpha}(\bar{\Omega}),\quad B_1=W^{-1,\frac{4}{3}}(\Omega)
	\end{align*}
	with $0<\alpha<1$. Since $B_0$ is compactly embedded into $B$, the uniform bounds \eqref{4.3}, \eqref{D2}, and \eqref{lemma4.2} imply that there exists a subsequence, still denoted by $h^\kappa$, and a function $h\in L^2(0,T;C^\alpha(\bar{\Omega}))$, such that $h^\kappa$ converges to $h$ strongly in $L^2(0,T;C^\alpha(\bar{\Omega}))$. By extracting a further subsequence if necessary, we obtain the pointwise-in-time convergence \eqref{5.2b}. This proves \eqref{5.2} and \eqref{5.2b}.
	
	\noindent Furthermore, choosing $p_0=\infty$, $p_1=\frac{4}{3}$, $B_0=H^1(\Omega)$, $B=L^2(\Omega)$, and $B_1=W^{-1,\frac{4}{3}}(\Omega)$, Lemma \ref{aubin} yields \eqref{5.3}.
	
	\noindent From the uniform bounds \eqref{4.3}, \eqref{D2}, and \eqref{lemma4.2}, we can extract weakly and weakly-$\ast$ converging subsequences. The weak lower semicontinuity of norms ensures that the limit functions inherit these bounds, which proves \eqref{5.1}. 
	
	\noindent The uniform $L^\infty(0,T; H^1(\Omega))$ bound implies that $\nabla h^\kappa$ is bounded in $L^2(Q_T)$, yielding \eqref{5.55}.
	
	\noindent Applying the Gagliardo-Nirenberg inequality \eqref{nirenberg66} to $h^\kappa - h$ and integrating over time, we obtain
	\begin{align}\label{666}
		\int_{0}^{T} \left\| h^\kappa-h \right\|_{L^6(\Omega)}^6 dt 
		&\leq C \int_{0}^{T} \left\| D_x^2(h^\kappa - h) \right\|^2 \left\| h^\kappa-h \right\|^4 dt + C \int_{0}^{T} \left\| h^\kappa-h \right\|^6 dt \nonumber\\
		&\leq C \left( \int_{0}^{T} \left\| D_x^2 h^\kappa - D_x^2 h \right\|^2 dt \right) \left(\sup_{t\in [0,T]} \left\| h^\kappa-h \right\|\right)^4 \nonumber\\
		&\quad + C T \left(\sup_{t\in [0,T]} \left\| h^\kappa-h \right\|\right)^6.
	\end{align}
	Since $D_x^2 h^\kappa$ is uniformly bounded in $L^2(Q_T)$ and $h^\kappa \to h$ strongly in $C([0,T]; L^2(\Omega))$ by \eqref{5.3}, the right-hand side of \eqref{666} tends to $0$ as $\kappa \to 0$. Thus, $h^\kappa \rightarrow h$ strongly in $L^6(Q_T)$.
	
	\noindent The algebraic inequality $\left| \sqrt{a^2+\kappa^2} - |b| \right| \leq |a-b| + \kappa$ yields
	\begin{align*}
		\left\| |h^\kappa|_{\kappa} - |h| \right\|_{L^6(Q_{T})} 
		&= \left\| \sqrt{|h^\kappa|^2+\kappa^2} - |h| \right\|_{L^6(Q_{T})} \nonumber\\
		&\leq \left\| h^\kappa - h \right\|_{L^6(Q_{T})} + \kappa |Q_T|^{\frac{1}{6}},
	\end{align*}
	where $|Q_T|$ denotes the Lebesgue measure of $Q_T$. Since $\left\| h^\kappa - h \right\|_{L^6(Q_{T})} \to 0$, we arrive at \eqref{kappa6}.
	
	\noindent Finally, algebraic factorization, H\"older's inequality and \eqref{kappa6} yield
	\begin{align*}
		&\left\| |h^\kappa|_{\kappa}^3-|h|^3 \right\|_{L^2(Q_{T})} \nonumber\\
		&\leq \left\| |h^\kappa|_\kappa - |h| \right\|_{L^6(Q_T)} \left\| |h^\kappa|_\kappa^2 + |h^\kappa|_\kappa |h| + |h|^2 \right\|_{L^3(Q_T)} \rightarrow 0
	\end{align*}
	and
	\begin{align*}
		\left\| |h^\kappa|_{\kappa}^2-|h|^2 \right\|_{L^3(Q_{T})} 
		&\leq \left\| |h^\kappa|_\kappa - |h| \right\|_{L^6(Q_T)} \left\| |h^\kappa|_\kappa + |h| \right\|_{L^6(Q_T)} \rightarrow 0,
	\end{align*}
	where the uniform bounds of the second factors in $L^3$ and $L^6$ are guaranteed by the $L^6(Q_T)$ bounds of $h^\kappa$ and $h$. This proves \eqref{5.4} and \eqref{5.5}, completing the proof of the lemma.
\end{proof}

The global third-order weak derivative of $h$ may not exist at $h=0$. Therefore, we establish a local convergence result for the nonlinear fluxes on the non-zero set of $h$. Define
\begin{align*}
	\mathcal{A}^h &:= \left\{(t,x) \in Q_{T} \mid |h(t,x)|>0\right\},
	\\
	\mathcal{A}^h(t) &:= \left\{x \in \Omega \mid (t,x) \in \mathcal{A}^h\right\}, \quad t \in (0,T).
\end{align*}
It follows from \eqref{5.2b} that $h(t)\in C^\alpha(\bar{\Omega})$ for almost all $t\in(0,{T})$, implying the spatial slice $\mathcal{A}^h(t)$ is an open set for almost all $t\in(0,T)$.

\begin{lema}\label{jubu}
	The limit function $h$ has the local weak $L^2$-derivative $\nabla\Delta h$  on $\mathcal{A}^h$ in the sense of Definition \ref{ruodaoshu}. Moreover, there exists a subsequence $h^\kappa$ such that
	\begin{align}\label{5.6}
		|h^\kappa|_{\kappa}^3\nabla \Delta h^\kappa\rightharpoonup \chi \quad \text{weakly in } L^{\frac{4}{3}}(Q_{T}),
	\end{align}
	where the limit function $\chi\in L^{\frac{4}{3}}(Q_{T})$ is given by
	\begin{align}\label{5.7}
		\chi(t,x)= \left \{
		\begin{aligned}
			&0, \quad &&\text{if } h(t,x)=0, \\
		    &|h|^3\nabla \Delta h, \quad &&\text{if } h(t,x)\neq 0.
		\end{aligned}\right.
	\end{align}
\end{lema}

\begin{proof}
To show that $h$ admits a local weak derivative $\nabla\Delta h$ on $\mathcal{A}^h$, we first construct the exhausting sets $\{\mathcal{A}_n\}_{n\in \mathbb{N}}$ required by Definition \ref{ruodaoshu}. Recall from \eqref{5.2b} that the measurable function $t \mapsto \|h^\kappa(t)-h(t)\|_{C^\alpha(\bar{\Omega})}$ converges to zero for almost all $t\in(0,T)$. Theorem \ref{shibian} implies that there exists an increasing sequence of measurable subsets $\{M_n\}_{n\in \mathbb{N}} \subset (0,T)$ satisfying
\begin{align}\label{24.13}
	|(0,T) \setminus M_n| \leq \frac{1}{n}, \quad M_n \subset M_{n+1},
\end{align}
such that $\|h^\kappa(t)-h(t)\|_{C^\alpha(\bar{\Omega})}$ converges to zero uniformly with respect to $t\in M_n$. This implies that $h^\kappa$ converges to $h$ uniformly on $M_n\times \Omega$. We define
\begin{align*}
	\hat{\mathcal{A}}_n = \left\{(t,x)\in Q_{T} \mid |h(t,x)|>\frac{1}{n}\right\}, \quad \mathcal{A}_n = \hat{\mathcal{A}}_n \cap (M_n\times \Omega).
\end{align*}
By \eqref{24.13}, we have $\mathcal{A}_n \subset \mathcal{A}_{n+1}$ and
\begin{align*}
	\bigcup_{n=1}^{\infty} \mathcal{A}_n = \left(\bigcup_{n=1}^{\infty}\hat{\mathcal{A}}_n\right) \setminus \bigcap_{n=1}^{\infty}\left( \left((0,T) \setminus M_n\right)\times \Omega \right) = \mathcal{A}^h \setminus (N\times \Omega),
\end{align*}
where $N = \bigcap_{n=1}^{\infty}((0,T) \setminus M_n)$ has zero Lebesgue measure. This implies that the spatial slices satisfy
\begin{align*}
	\mathcal{A}^h(t) = \bigcup_{n=1}^\infty \mathcal{A}_n(t)
\end{align*}
for $t\in(0,T) \setminus N$. Since $h(t)$ is continuous for $t\in M_n$, the set $\hat{\mathcal{A}}_n(t)$ is open in $\Omega$. Therefore, $\mathcal{A}_n(t)$ and $\mathcal{A}^h(t)$ are open subsets of $\Omega$ for a.e. $t \in (0,T)$.

We next show that $\nabla\Delta h^\kappa$ is locally bounded on $\mathcal{A}_n$. The uniform convergence of $h^\kappa$ on $\mathcal{A}_n$ implies that there exists $\kappa_0 > 0$ such that for all $0<\kappa<\kappa_0$ and all $(t,x)\in \mathcal{A}_n$, we have $|h^\kappa| > \frac{1}{2n}$. Invoking \eqref{guji32}, we estimate
\begin{align*}
	C \geq \int\int_{Q_T} |h^\kappa|_{\kappa}^3|\nabla\Delta h^\kappa|^2 dxdt \geq \iint_{\mathcal{A}_n} |h^\kappa|^3|\nabla\Delta h^\kappa|^2 dxdt \geq \frac{1}{8n^3} \iint_{\mathcal{A}_n} |\nabla\Delta h^\kappa|^2 dxdt.
\end{align*}
This yields $\left\|\nabla\Delta h^\kappa\right\|_{L^2(\mathcal{A}_n)} \leq \sqrt{8Cn^3}$. By the Eberlein-Šmulian theorem, we can extract a subsequence such that
\begin{align}\label{5.13}
	\nabla\Delta h^\kappa \rightharpoonup g_n \quad \text{weakly in } L^2(\mathcal{A}_n).
\end{align}

We now verify that $g_n(t) = \nabla \Delta h(t)$ in the distributional sense on $\mathcal{A}_n(t)$ for a.e. $t$. Define the separable space
\begin{align*}
	L^2_{\mathcal{A}_n}(0,T;H_0^3(\Omega)) = \left\{v\in L^2(\mathcal{A}_n) \mid v(t)\in H_0^3(\mathcal{A}_n(t)) \text{ for a.e. } t\in (0,T)\right\},
\end{align*}
and let $\mathcal{K}$ be a countable dense subset thereof. For $\varphi\in \mathcal{K}$ and $(t, t+t_0) \subset (0,T)$, integration by parts yields
\begin{align*}
	\int_{t}^{t+t_0} \int_{\mathcal{A}_n(\tau)} g_n \varphi \,dxd\tau
	&= \lim_{\kappa \rightarrow 0} \int_{t}^{t+t_0} \int_{\mathcal{A}_n(\tau)} \nabla\Delta h^\kappa \varphi \,dxd\tau \nonumber\\
	&= - \lim_{\kappa \rightarrow 0} \int_{t}^{t+t_0} \int_{\mathcal{A}_n(\tau)} h^\kappa \nabla\Delta \varphi \,dxd\tau \nonumber\\
	&= - \int_{t}^{t+t_0} \int_{\mathcal{A}_n(\tau)} h \nabla\Delta \varphi \,dxd\tau.
\end{align*}
Dividing by $t_0$ and taking $t_0 \rightarrow 0$, the Lebesgue Differentiation Theorem asserts that for a.e. $t \in (0,T)$,
\begin{align}\label{5.15}
	\int_{\mathcal{A}_n(t)} g_n(t) \varphi(t) dx = - \int_{\mathcal{A}_n(t)} h(t) \nabla \Delta\varphi(t) dx.
\end{align}
By a standard density argument, \eqref{5.15} holds for all test functions in $H_0^3(\mathcal{A}_n(t))$ for a.e. $t \in M_n$. Thus $h(t) \in H^3(\mathcal{A}_n(t))$ and
\begin{align}\label{5.16}
	\nabla \Delta h(t) = g_n(t) \in L^2(\mathcal{A}_n(t)) \quad \text{for a.e. } t \in M_n.
\end{align}

We then patch these local derivatives together. Define $g$ on $\mathcal{A}^h$ by setting $g|_{\mathcal{A}_n} = g_n$. This is well-defined because the uniqueness of weak derivatives ensures $g_m|_{\mathcal{A}_n} = g_n$ for $m \geq n$. A standard compactness argument on any compact subset $K \subset \mathcal{A}^h(t)$ demonstrates that $K \subset \mathcal{A}_n(t)$ for sufficiently large $n$, which verifies $g(t) \in L^2_{\mathrm{loc}}(\mathcal{A}^h(t))$. Consequently, $g(t)$ is the third-order spatial derivative of $h(t)$ on $\mathcal{A}^h(t)$. Thus, $h$ possesses the local weak $L^2$-derivative $\nabla\Delta h$ on $\mathcal{A}^h$ according to Definition \ref{ruodaoshu}.

Turning to the nonlinear flux, the uniform bound \eqref{guji43} guarantees the existence of a weak limit $\chi \in L^{\frac{4}{3}}(Q_T)$ such that $|h^\kappa|_{\kappa}^3\nabla \Delta h^\kappa \rightharpoonup \chi$.
On $\mathcal{A}_n$, the uniform convergence of $|h^\kappa|_\kappa^3 \to |h|^3$ with \eqref{5.13} and \eqref{5.16} implies
\begin{align*}
	|h^\kappa|_{\kappa}^3\nabla\Delta h^\kappa \rightharpoonup |h|^3\nabla\Delta h \quad \text{weakly in } L^{\frac{4}{3}}(\mathcal{A}_n).
\end{align*}
By the uniqueness of weak limits, $\chi = |h|^3\nabla\Delta h$ a.e. on $\mathcal{A}_n$, and thus a.e. on $\mathcal{A}^h$.

To determine $\chi$ on the set $\{h=0\}$, we consider the set $\{|h| \leq \delta\}$ for any $\delta>0$. Applying H\"older's inequality and \eqref{guji32}, we estimate
\begin{align}\label{5.27}
	\iint_{\{|h|\leq\delta\}} |h^\kappa|_{\kappa}^4 |\nabla\Delta h^\kappa|^{\frac{4}{3}} dxdt
	&\leq \left( \iint_{\{|h|\leq\delta\}} |h^\kappa|_{\kappa}^6 dxdt \right)^{\frac{1}{3}} \left( \iint_{\{|h|\leq\delta\}} |h^\kappa|_{\kappa}^3 |\nabla\Delta h^\kappa|^2 dxdt \right)^{\frac{2}{3}} \nonumber\\
	&\leq C \left\| |h^\kappa|_{\kappa} \right\|_{L^6(\{|h|\leq\delta\})}^2 \nonumber\\
	&\leq C \left( \left\| |h^\kappa|_{\kappa}-|h| \right\|_{L^6(Q_{T})} + \left\| h \right\|_{L^6(\{|h|\leq\delta\})} \right)^2 \nonumber\\
	&\leq C \left( \left\| |h^\kappa|_{\kappa}-|h| \right\|_{L^6(Q_{T})} + \delta |Q_T|^{\frac{1}{6}} \right)^2.
\end{align}
Passing to the limit $\kappa \to 0$ in \eqref{5.27}, the strong convergence \eqref{kappa6} implies the first term on the right-hand side vanishes. Invoking the weak lower semicontinuity of the $L^{\frac{4}{3}}$-norm for $|h^\kappa|_{\kappa}^3\nabla\Delta h^\kappa \rightharpoonup \chi$, we obtain
\begin{align*}
	\|\chi\|_{L^{\frac{4}{3}}(\{h=0\})}^{\frac{4}{3}} \leq \|\chi\|_{L^{\frac{4}{3}}(\{|h|\leq \delta\})}^{\frac{4}{3}} \leq \liminf_{\kappa \to 0} \iint_{\{|h|\leq\delta\}} \left| |h^\kappa|_{\kappa}^3\nabla\Delta h^\kappa \right|^{\frac{4}{3}} dxdt \leq C \delta^2 |Q_T|^{\frac{1}{3}}.
\end{align*}
Taking $\delta \to 0$ yields $\|\chi\|_{L^{\frac{4}{3}}(\{h=0\})} = 0$, which completes the proof of the lemma.
\end{proof}

We now establish the strict positivity of the limit function $h$. The entropy estimate in Lemma $\ref{entropy}$ prevents the film thickness $h$ from becoming negative or touching zero on sets of positive measure.

\begin{lema}\label{zhengxing}
	There exists a constant $C>0$ independent of $\kappa$, such that for any initial data $h_0$ satisfying \eqref{jiashe}, the following properties hold:
	\begin{align}
		\label{zheng}
		& h(t,x) > 0 \quad \text{for a.e. } x \in \Omega, \text{ for all } t \in [0,T], \\
		\label{daoshu}
		& \left\| 1/h \right\|_{L^\infty(0,T;L^1(\Omega))} \leq C.
	\end{align}
\end{lema}

\begin{proof}
	Recall from \eqref{6.15} that $\int_{\Omega} G_\kappa(h^\kappa(t,x)) dx \leq C$ for all $t \in [0,T]$. First, the pointwise-in-time strong convergence \eqref{5.2b} implies that there exists a subset $\mathcal{T} \subset [0,T]$ of full measure such that $h^\kappa(t,x) \to h(t,x)$ uniformly in $x$ for all $t \in \mathcal{T}$. Note that the limit entropy satisfies $G_0(s) \geq -1/A$ for all $s > 0$. Applying Fatou's lemma to the non-negative sequence $G_\kappa(h^\kappa) + 1/A \geq 0$, we deduce that for any $t \in \mathcal{T}$
	\begin{align}\label{Fatou_G}
		\int_{\Omega} G_0(h(t,x)) dx \leq \liminf_{\kappa \to 0} \int_{\Omega} G_\kappa(h^\kappa(t,x)) dx \leq C.
	\end{align}
	
	Next, we extend this bound to all $t \in [0,T]$ by utilizing \eqref{5.3}, which asserts $h \in C([0,T]; L^2(\Omega))$. For any $t \in [0,T]$, we can choose a sequence $\{t_n\}_{n=1}^\infty \subset \mathcal{T}$ such that $t_n \to t$. The temporal continuity implies that $h(t_n) \to h(t)$ strongly in $L^2(\Omega)$. Extracting a further subsequence yields $h(t_n, x) \to h(t, x)$ for a.e. $x \in \Omega$. Since $t_n \in \mathcal{T}$, we conclude from \eqref{Fatou_G} that $\int_{\Omega} G_0(h(t_n,x)) dx \leq C$ for every $n$. Applying Fatou's lemma a second time to the non-negative sequence $G_0(h(t_n,x)) + 1/A \geq 0$, we obtain
	\begin{align}\label{Fatou_t}
		\int_{\Omega} G_0(h(t,x)) dx \leq \liminf_{n \to \infty} \int_{\Omega} G_0(h(t_n,x)) dx \leq C \qquad \text{for all } t \in [0,T].
	\end{align}
	Since the spatial integral of $G_0(h(t,x))$ is finite for every $t \in [0,T]$, $G_0(h(t,x))$ is finite a.e. in $\Omega$. Definition \eqref{g} specifies $G_0(s) = \infty$ for all $s \leq 0$. Thus $h(t,x) > 0$ a.e. in $\Omega$ for all $t \in [0,T]$. This establishes \eqref{zheng}.
	
	Finally, Definition \eqref{g} yields $G_0(s) \geq \frac{1}{2s} - \frac{1}{A}$ for $s>0$. Substituting into \eqref{Fatou_t} yields
	\begin{align*}
		\int_{\Omega} \frac{1}{2h(t,x)} dx - \frac{|\Omega|}{A} \leq \int_{\Omega} G_0(h(t,x)) dx \leq C \qquad \text{for all } t \in [0,T].
	\end{align*}
	Rearranging this inequality, we obtain
	\begin{align*}
		\int_{\Omega} \frac{1}{h(t,x)} dx \leq 2C + \frac{2|\Omega|}{A} := C_2 \qquad \text{for all } t \in [0,T].
	\end{align*}
	Taking the essential supremum over $t \in (0,T)$ leads to \eqref{daoshu}. The proof of the lemma is complete.
\end{proof}

\noindent\textbf{Proof of Theorem $\textbf{\ref{jieguo}}$.}
Let $h^\kappa$ be the corresponding sequence of weak solutions to \eqref{a1}-\eqref{a3}.
In Lemma $\ref{shoulian1}$, there exists a subsequence, still denoted by $h^\kappa$, that converges to $h$. We shall show that $h$ is a weak solution of the problem \eqref{1}-\eqref{3} and satisfies \eqref{2.8}-\eqref{2.10}. To this end recalling \eqref{4.3}, \eqref{D2} and \eqref{lemma4.2} we obtain \eqref{2.8} and \eqref{2.9} for $h$. Relation \eqref{2.10} is implied by Lemma $\ref{jubu}$. Thus from \eqref{2}, \eqref{2.8} and the Gagliardo-Nirenberg inequality \eqref{nirenberg66} we get \eqref{d1}. The relations \eqref{111}-\eqref{112} follow from Lemma $\ref{zhengxing}$. To prove that \eqref{d2} holds, multiplying \eqref{a1} by a test function $\varphi\in C^\infty_{0}((-\infty,T)\times \mathbb{R})$ and integrating the resulting equation over $Q_{T}$, we obtain
\begin{align*}
	&(h^\kappa,\varphi_t)_{Q_{T}}+\alpha_{1} ( |h^\kappa|_{\kappa}^3 \nabla \Delta  h^\kappa,\nabla \varphi)_{\mathcal{A}^{h}} 
	\nonumber\\
	&=\alpha_{2} (|h^\kappa|_{\kappa}^3 \nabla h^\kappa, \nabla \varphi)_{Q_{T}} - \alpha_{3} (|h^\kappa|_{\kappa}^2 \nabla h^\kappa, \nabla \varphi)_{Q_{T}}-(h_0,\varphi(0))_{\Omega}.
\end{align*}
The equality $(\ref{d2})$ follows from this relation if we prove that
\begin{align} \label{5.29}
	(h^\kappa,\varphi_t)_{Q_{T}} \quad &\rightarrow \quad (h,\varphi_t)_{Q_{T}},
	\\ \label{5.30}
	( |h^\kappa|_{\kappa}^3 \nabla \Delta  h^\kappa,\nabla \varphi)_{Q_T} \quad &\rightarrow \quad ( h^3\nabla\Delta h,\nabla\varphi)_{\mathcal{A}^{h}} ,
	\\ \label{5.31}
	(|h^\kappa|_{\kappa}^3 \nabla h^\kappa, \nabla \varphi)_{Q_{T}} \quad &\rightarrow \quad (h^3 \nabla h, \nabla\varphi)_{Q_{T}},
	\\ \label{5.32}
	(|h^\kappa|_{\kappa}^2 \nabla h^\kappa, \nabla \varphi)_{Q_{T}} \quad &\rightarrow \quad (h^2 \nabla h, \nabla\varphi)_{Q_{T}},
\end{align}
for $\kappa\rightarrow 0$. Now, the relation \eqref{5.29} is a consequence of \eqref{5.3} and the relation \eqref{5.30} follows from Lemma $\ref{jubu}$. By \eqref{5.4}, \eqref{5.5} and \eqref{5.55} we assert that \eqref{5.31} and \eqref{5.32} hold. Therefore, together with Lemma $\ref{jubu}$, $h$ is a weak solution to the initial-boundary value problem \eqref{1}-\eqref{3} having the regularity properties stated in Theorem $\ref{jieguo}$. The proof of this theorem is complete.

\vskip0.5cm

\noindent\textbf{Data availability}: No data was used for the research described in the article.

\noindent\textbf{Confict of interest}: The authors declare that they have no confict of interest.

\end{document}